\documentclass{amsart}
\usepackage{amsthm} 
\usepackage{amstext}
\usepackage{amsbsy}
\usepackage{mathrsfs}
\usepackage[all]{xy}
\usepackage{amscd}
\usepackage{amsmath}
\usepackage{amsfonts}
\usepackage{hyperref}
\usepackage{latexsym}
\usepackage{amssymb}
\usepackage{fancyhdr}
\usepackage{times}
\usepackage{pst-all}
\usepackage{pst-tree}
\usepackage{textcomp}
\usepackage [T1] {fontenc}
\usepackage{graphicx}
\usepackage{pst-plot}
\usepackage[graphics,floats,textmath,displaymath,delayed]{preview}
\usepackage{url}

\newtheorem{theorem}{Theorem}
\newtheorem{lemma}[theorem]{Lemma}
\newtheorem{proposition}{Proposition}
\newtheorem{corollary}[theorem]{Corollary}

\theoremstyle{definition}
\newtheorem{definition}[theorem]{Definition}

\theoremstyle{remark}
\newtheorem{remark}[theorem]{Remark}

\numberwithin{equation}{section} 

\newcommand{\R}{\mathbb{R}}

\newcommand{\N}{\mathbb{N}}

\newcommand{\Z}{\mathbb{Z}}

\newcommand{\PP}{\mathbb{P}}

\newcommand{\A}{\mathcal{A}}

\newcommand{\code}[7]{\begin{array}{|c|c|c|c|}
\cline{1-3}  #1 & #2 & #3 & \multicolumn{1}{c}{$$} \\\hline #4 & #5 & #6 & #7 \\\hline
\end{array} }
\newcommand{\astc}{\textasteriskcentered}
\newcommand{\rowone}{1&1&1&1&1&0&0&0&0&0&0&0&0&0&0&0}
\newcommand{\rowtwo}{0&0&0&0&0&1&1&1&1&1&0&0&0&0&0&0}
\newcommand{\rowfive}{0&0&0&0&0&0&0&0&0&0&1&1&1&0&0&0}
\newcommand{\rowsix}{0&0&0&0&0&0&0&0&0&0&0&0&0&1&1&1}

\begin{document}

\begin{center}

\normalsize{\textbf{HYPERBOLICITY OF THE TRACE MAP FOR A STRONGLY COUPLED QUASIPERIODIC SCHR\"ODINGER OPERATOR}}

\vspace{1cm}

\footnotesize{

\textrm{EMILIANO DE SIMONE, LAURENT MARIN} \\

\vspace{0.2cm}

\textit{Department of Mathematics, University of Helsinki, P.O.Box 68 (Gustaf H\"allstr\"omin katu 2{b}) \\
Helsinki, 00014 , Finland \\ emiliano.desimone@helsinki.fi \\ laurent.marin@univ-orleans.fr} 
}

\end{center}
\vspace{0.5cm}

\footnotesize{ 

\begin{center} \emph{Abstract}
\end{center}
\vspace{0.2cm}

\begin{center}
\parbox[c]{300pt}{We consider the trace map associated with the silver ratio Schrödinger operator as a diffeomorphism on the invariant surface associated with a given coupling constant and prove that the non-wandering set of this map is hyperbolic if the coupling is sufficiently large. As a consequence, for this values of the coupling constant, the local and global Hausdorff dimension and the local and global box counting dimension of the spectrum of this operator all coincide and are smooth functions of the coupling constant.}
\end{center}
\vspace{0.4cm}


\vspace{0.4cm}

Mathematics Subject Classification 2000: 

}

\normalsize

\section{Introduction}

Consider the discrete Schrödinger operator associated to a Sturmian quasiperiodic potential:
\begin{equation}
(H \psi ) (n) = \psi(n-1) - \psi(n+1) + v(n, \theta)\psi(n)
\end{equation}

The quasi periodicity of $v(n)$ is generated by an irrational number $\theta$ according to the formula
\begin{equation}
v(n, \theta) = V( [ (n+1) \theta] - [ n \theta] ) \label{potential}
\end{equation}
where $V \in \mathbb R^+$ is the coupling constant and for $x \in \mathbb R$, $[x]$ denotes the largest integer smaller than $|x|$.

This operator family, who describes electrical properties of quasicrystals, exhibits a number of interesting phenomena, such as Cantor spectrum of Lebesgue measure zero \cite{s1,BIST1989} and purely singular continuous spectral measure \cite{DamLenz}. Moreover, it was recently shown that is also gives rise to anomalous transport for a large class of irrational number \cite{dam1, LaurentThesis}.

We restrict our attention to the potential associated to the irrational number $\omega$, the so-called silver ratio, whose continued fraction expansion is given by
\begin{equation}
\omega = \frac{1}{2+\frac{1}{2+ \frac{1}{2 + \cdots}}} = [0,2, 2, \ldots 2, \ldots].
\end{equation}

We shall show that this restriction, we can study the spectrum of $H$ by means of an auxilliairy  dynamical system, described by the so-called trace map
\begin{equation}
T(x, y, z) := (x(y^2-1) - zy, xy-z, y ), 
\end{equation}

In great generality, the properties of the trace map are known to be closely related to all the spectral properties mentioned above \cite{s1,BIST1989,Raymond,dam1}. The expression of the trace map and its connection to spectral properties of the operators is a consequence of the quasiperiodic structure of the  potential (e.g. \cite{Lot} for more details on quasiperiodic matters). 
  
We shall show that the spectrum of $H$ can be determined by studying the non wandering set of $T$ on a invariant surface $S_V$ that will be introduced later on. Since the spectrum is of Lebesgue measure zero, it is logical to investigate its fractal dimension. Moreover, its fractal dimension are known to be linked with spectral and dynamical properties \cite{tche2,Raymond,chin1,LaurentThesis}.
This study has been  already done in the case where the operator is the so-called Fibonacci Hamiltonian for which the non-wandering set of $T$ on $S_V$ can be proven to be hyperbolic for $V>16$ \cite{cas} and for small coupling \cite{damgorod}. For this values of coupling, the general theory of hyperbolic surface diffeomorphisms yields  the exact asymptotic behavior  of the fractal dimension of the spectrum of the Fibonacci Hamiltonian as a function of $V$\cite{tche2,damgorod,pesin,manning,palis}.

The idea of this paper is to provide the same analysis for the  Schrödinger operator associated to the silver ratio and prove the hyperbolicity of the trace map, for large coupling.  We obtain then the same properties for the fractal dimension of the spectrum, that is coincidence of local and global Hausdorff and box counting dimension. Moreover all these dimensions are smooth function of the coupling constant $V$. The interest in considering the silver ratio Schrodinger operator is that this model provides the simpliest case where the pseudo spectrum and the so-called auxilliairy spectrum do not coincide. This fact yields  a much more complicated pattern when one studies the trace map as a dynamical system.

After this paper was finished we learned that Serge Cantat provided a proof of uniform hyperbolicity of the trace map for all non zero values of the coupling and all quadratic irrational numbers \cite{Cantat}. Our results were obtained independently and we use completely different methods.

\section{Description of trace map and statements of the  results}

The main tool that we shall use here is the trace map.  Let us recall how it arises from the structure of the quasi periodical potential (see \cite{s1,BIST1989,tes} for detailed proofs of the statements below).

Considering the difference equation,
\begin{equation}\label{freeequation}
H \psi = E \psi  
\end{equation}
one can see any solution of (\ref{freeequation}) fulfills the following relation:
\begin{equation}
\psi_{n+1} = T_n(V,\omega,E) \psi_{n}  \quad \text{for all} \quad n \in \mathbb N \label{transfer}
\end{equation}

where the transfer matrix $T_n$ is given by
\begin{equation}
T_n(V,\omega,E) = \left(\begin{array}{cc} E - V v(n, \omega) & -1 \\1 & 0\end{array}\right).
\end{equation}

Denote the Pell number sequence by $\{F_m\}_{m \in \mathbb N}$ defined by the relation 
\begin{equation}
\left\{\begin{array}{l} \label{relation}
 F_{-1}=0 \; , \; F_0 = 1 \\
F_{m+1} = 2F_m + F_{m-1},
\end{array}\right.
\end{equation}
Then, one can show that

 $M_m := T_{F_m}(V,\omega,E)T_{F_m -1}(V,\omega,E) \cdots T_1(V,\omega,E)$, obey the following renormalization relation for any $m\geq 0$
\begin{equation}\label{renormalization}
M_{m+1} = M_{m}^2M_{m-1}, \quad M_{-1} = \left(\begin{array}{cc}1 & -V \\0 & 1\end{array}\right) \quad M_0 = \left(\begin{array}{cc}E & -1 \\1 & 0\end{array}\right).
\end{equation}

It follows that defining $u_m := \text{Tr} (M_{m-1}M_m)$ and $v_m := \text{Tr} (M_m)$, the two infinite sequences $\{u_m\}_{m \in \N} , \{v_m\}_{m \in \N}$ obey the following recurrence relations
\begin{equation} \label{recurrence}
\left \{ \begin{array}{ll} 
u_{m+1} = u_m(v_m^2-1) - v_{m-1}v_m , & u_0 = E-V\\ 
v_{m+1} = u_mv_m-v_{m-1}  , & v_{-1} = 2 \: , \: v_0 = E,
\end{array}
\right.
\end{equation}

who have a conserved quantity, precisely 
\begin{equation}\label{invariant}
u_m^2 + v_m^2 +v_{m-1}^2 - u_m v_m v_{m-1} = V^2 +4, \forall m.
\end{equation}

The renormalization map is then defined by
\begin{equation}
T(x, y, z) := (x(y^2-1) - zy, xy-z, y ), \label{RG}
\end{equation}
and it acts by shifting $u_n$ and $v_n$ in the following way:
\begin{equation} \label{shift1}
T : (u_m, v_m , v_{m-1}) \mapsto (u_{m+1}, v_{m+1}, v_m).
\end{equation}

We should remark at this point, that if one considers the Fibonacci Hamiltonian then one has $u_m = v_{m+1}$ for all $m$. The auxilliairy spectrum defined from the sequence $\{u_m\}_m$ coincide with the pseudo spectrum defined with the sequence $\{v_m\}_m$. This particular fact  yields to a two-coordinates shifting trace map. The trace map of silver ratio Schrödinger operator shifts only one coordinate, namely $y$.  This general pattern implies a  more complicated combinatorics for the dynamical system associated to $T$.

The relation \ref{invariant} implies  that $T$ preserves the family of cubic surfaces :
\begin{equation}
S_V = \{ (x,y,z) \in \Rö^{3} \, |\, x^2 + y^2 + z^2 -xyz = 4 + V^2 \}.
\end{equation}

Thus, defining the line in $\R^3$, $\ell_V(E) := (E-V,E,2)$, we can express the spectrum of $H$ in terms of $T$:
\begin{equation} \label{spectrum}
\sigma(H) = \{E \: | \:  \lim_{n \rightarrow \infty } \pi_{y} T^n(\ell_V(E)) < \infty \},
\end{equation}
where $\pi_y : (x,y,z) \mapsto y$ denotes the projection onto the second coordinate axis. 

This is a simple rewriting of the coincidence between the spectrum and the pseudo spectrum proven by Süt\H{o} in \cite{s1} and by Bellissard et al. in \cite{BIST1989}.

We notice that for all $V \in \R$ the line $\ell_V$ is contained in the $T$-invariant surface $S_V$
We shall denote $T_{V} = T{|_{S_{V}}}$ and in order to derive results on the spectrum of $H$ we shall study the non-wandering set $\Sigma_V$ of $T_V$ and prove that it is a Cantor set with an hyperbolic structure.

We recall that an invariant closed set $\Omega$ of a diffeomophism $f: M \to M$ is hyperbolic if there exists a splitting of the tangent space $T_xM = E^u_x \oplus E^s_x$ at every point $x \in \Omega$ such that the splitting is invariant under $Df$ and the differential $Df$ exponentially contracts vectors from the stable subspaces $\{E_x^s\}$ and exponentially expands vectors from the unstable subspaces $E^u_x $. A hyperbolic set $\Omega$ is locally maximal if there exists a neighborhood $U(\Omega)$ such that
$$
\Omega = \bigcap_{n \in \Z} f^n(U).
$$

\begin{theorem}\label{thprincipal}
For $V$ sufficiently large, we have that the non wandering set $\Omega_V$ of $T_V$ is a locally maximal hyperbolic set that is homeomorphic to a Cantor set.
\end{theorem}

\subsection{Some Properties of Locally Maximal Hyperbolic Invariant sets of surface Diffeomorphisms.}

Given Theorem \ref{thprincipal}, several general results apply to the trace map of the strongly coupled silver ratio Schr\"{o}dinger operator. This section recalls as it was previously done in [\cite{damgorod,tche2}],  some of these results that yield interesting spectral consequences, which are discussed below.

Consider a locally maximal invariant transitive hyperbolic set $\Lambda \subset M$, $\dim M = 2$, of a diffeomorphism $f \in \text{Diff}^r (M), r \geq 1$. We have $\Lambda = \bigcap_{n \in \Z} f^n(U(\Lambda))$ for some neighborhood $U(\Lambda)$. Assume also that $\dim E^u = \dim E^s =1$. Then, the following properties hold.

\subsubsection{Stability}

There is a neighborhood $\mathcal{U} \subset \text{Diff}^1 (M)$ of the map $f$ such that for every $g \in \mathcal{U}$, the set $\Lambda_g =  
\bigcap_{n \in \Z} g^n(U(\Lambda))$ is a locally maximal invariant hyperbolic set of $g$. Moreover, there is a homeomorphism $h: \Lambda \to \Lambda_g$ that conjugates $f|_\Lambda$ and $g|_{\Lambda_g}$, that is, the following diagram commutes:

$$
\xymatrix{
\Lambda \ar[d]_h \ar[r]^{f|_\Lambda} &  \Lambda \ar[d]^h \\
\Lambda_g \ar[r]^{g|_{\Lambda_g}} & \Lambda_g
}
$$

\subsubsection{Invariant Manifolds}
For $x \in \Lambda$ and small $\varepsilon >0$, consider the local stable and unstable sets
$$
W_\varepsilon^s(x) = \{w \in M : d(f^n(x), f^n(w)) \leq \varepsilon \hspace{2mm} \text{for all}  \hspace{2mm} n \geq 0 \},
$$ 
$$
W_\varepsilon^u(x) = \{w \in M : d(f^n(x), f^n(w)) \leq \varepsilon \hspace{2mm} \text{for all}  \hspace{2mm} n \leq 0 \}.
$$ 

If $\varepsilon >0$ is small enough, these sets are embedded $C^r$-disks with $T_xW_\varepsilon^s (x) = E^s_x$ and $T_xW_\varepsilon^u (x) = E^u_x$.
Define the (global) stable and unstable sets as
$$
W^s(x) = \bigcup_{n \in \N } f^{-n} (W_\varepsilon^s(x)),
\hspace{2mm}W^u(x) = \bigcup_{n \in \N } f^{n} (W_\varepsilon^u(x)).
$$
Define also
$$
W^s(\Lambda) =\bigcup_{x \in \Lambda }W^s(x)\hspace{2mm} \text{ and }
\hspace{2mm} W^u(\Lambda) =\bigcup_{x \in \Lambda }W^u(x).
$$

\subsubsection{Invariant Foliations}
A stable foliation for $\Lambda$ is a foliation $\mathcal{F}^s $ of a neighborhood of $\Lambda$ such that 
\begin{itemize}
\item[(1)] For each $x \in \Lambda, \mathcal{F}^s(x)$, the leaf containing $x$, is tangent to $E^s_x$,
\item[(2)] For each $x$ sufficiently close to $\Lambda$,  $f( \mathcal{F}^s(x)) \subset  \mathcal{F}^s(f(x))$.
\end{itemize}

An unstable foliation $\mathcal{F}^u$ can be defined in a similar way.

For a locally maximal hyperbolic set $\Lambda \subset M$ of a $C^1$-diffeomorphism $f: M \to M$, $\dim M= 2$, stable and unstable $C^0$ foliations with $C^1$-leaves can be constructed \cite{melo}.
In the case of $C^2$-diffeomorphism, $C^1$ invariant foliations exist (see, for example, \cite{palis2}, theorem 8 in Appendix 1).

\subsubsection{Local Hausdorff  and Box counting Dimension}

Consider, for $x \in \Lambda$ and small $\varepsilon > 0 $, the set $W^u_\varepsilon(x) \cap \Lambda$. Its Hausdorff dimension does not depend on $x \in  \Lambda $ and $\varepsilon >0$, and coincides with its local box counting dimension (see \cite{manning},\cite{takens}): 
$$
\dim_H W^u_\varepsilon(x) \cap \Lambda=\dim_B W^u_\varepsilon(x) \cap \Lambda.
$$
In a similar way,
$$
\dim_H W^s_\varepsilon(x) \cap \Lambda=\dim_B W^s_\varepsilon(x) \cap \Lambda.
$$
Denote $h^s = \dim_H W^s_\varepsilon(x) \cap \Lambda$ and 
$h^u = \dim_H W^u_\varepsilon(x) \cap \Lambda$. We will say that $h^s$ and $h^u$ are local stable and unstable Hausdorff dimensions of $\Lambda$.

For properly chosen small $\varepsilon >0 $, the sets 
$W^u_\varepsilon(x) \cap \Lambda$ and $W^s_\varepsilon(x) \cap \Lambda$
 are dynamically defined Cantor sets (see \cite{palis3} for definition and proof), and this implies, in particular, that
$$
h^s <1 \hspace{2mm} \text{and} \hspace{2mm} h^u <1,
$$
see, for example, Theorem 14.5 in  \cite{pesin}.

\subsubsection{Global Hausdorff Dimension} The Hausdorff dimension of 
$\Lambda$ is equal to its box counting dimension, and 
$$
\dim_H \Lambda =\dim_B \Lambda = h^s + h^u,
$$
see (\cite{manning,palis}).

\subsubsection{Continuity of the Hausdorff Dimension}

The local Hausdorff dimensions $h^s(\Lambda)$ and $h^u(\Lambda)$ depend continuously on $f: M \to M$ in the $C^1$-topology; see (\cite{manning,palis}). Therefore, $\dim_H \Lambda_f = \dim_B \Lambda_f = h^s(\Lambda_f) + h^u(\Lambda_f)$ also depends continuously on $f$ in the $C^1$-topology. Moreover, for a $C^r$ diffeomorphism $f: M \to M, r\geq 2$, the Hausdorff dimension of a hyperbolic set $\Lambda_f$ is a $C^{r-1}$ function of $f$, see (\cite{mane}).

\begin{remark}
For hyperbolic sets in dimension greater than two, many of these properties do not hold in general, see \cite{pesin} for more details.
\end{remark}

\subsection{Implications for the trace map and the spectrum}

Due to theorem \ref{thprincipal}, for $V$ sufficiently large, all the properties from the previous subsection can be applied to the hyperbolic set $\Omega_V$ of the trace map $T_V : S_V \to S_V$.

We have the following statement
\begin{lemma}
For $V$ large enough, and every $x \in \Omega_V$, the stable manifold
$W^s(x)$ intersects the line  $\ell_V$ transversally. 
\end{lemma}

The existence of a $C^1$-foliation $\mathcal{F}^s$ allows one to locally consider the set $W^s(\Omega_V) \cap \ell_V$ as a $C^1$-image of the set $W^u_\varepsilon(x) \cap \Omega_V$. Due to relation \ref{spectrum}, this implies the following properties of the spectrum $\sigma(H)$ for sufficiently strong coupling:

\begin{corollary}
For  $V$ sufficiently large, we have the following properties:

(1) The spectrum of $H$, $\sigma(H)$ depends continuously on  $V$ with respect to the Hausdorff metric.

(2) We have $dim_H(\sigma (H))=dim_B(\sigma (H))$.

(3) For all sufficiently small $\varepsilon >0$ and all $E \in 
\sigma (H)$, one has
$$
dim_H ([E-\varepsilon,E+\varepsilon]\cap\sigma (H) = \dim_H (\sigma (H)).
$$
et 
$$
dim_B ([E-\varepsilon,E+\varepsilon]\cap\sigma (H) = \dim_B (\sigma (H)).
$$

(4) The Hausdorff dimension  $\dim_H (\sigma (H))$ is a  $C^\infty$-function of $V$.
\end{corollary}

\subsection{Notation}
In order to study $\Sigma_V$ we shall first introduce some notation, and make use of the following relation between $T$ and its inverse:
\begin{equation}\label{conj}
T_V^{-1}(x,y,x) = (yz-x,z,y(z^2-1)-xz) = \rho \circ T_V \circ \rho (x,y,z),
\end{equation} 
where $\rho(x,y,z) = (y,z,x)$.
We define ${\bf w}(x,y,z)$ to be the seven elements table defined as follows
\begin{equation} \label{code}
{\bf w}(x,y,z) = \begin{array}{|c|c|c|c|}
\cline{1-3}  x(y^2 -1) - zy & x & yz-x & \multicolumn{1}{c}{$$} \\\hline xy-z & y & z & y(z^2 -1) - xz \\\hline 
\end{array}
\end{equation}
Let us clarify this definition: if we take the initial condition $(u_0, v_0, v_{-1}) = (x,y,z)$ to define two bi-infinite sequences by the relation \eqref{recurrence}, we get
\begin{equation}
{\bf w}(x,y,z) = \code{u_{1}}{u_{0}}{u_{-1}}{v_{1}}{v_0}{v_{-1}}{v_{-2}} \;.
\end{equation}
Thus, in view of \eqref{shift1} and \eqref{RG}, the three (overlapping) L-shaped subtables $\begin{array}{|c|c|}
\cline{1-1}   & \multicolumn{1}{c}{$$} \\\hline  &  \\\hline
\end{array} $
contained in the table in the r.h.s of \eqref{code} denote, respectively, from left to right:
\begin{equation} 
T_V(x,y,z) , \quad (x,y,z), \quad  T_V^{-1}(x,y,z),
\end{equation}
hence ${\bf w}(x,y,z)$ encodes informations about the "past" and the "future" of the point $(x,y,z)$ under the action of $T$.

To exploit the notation introduced so far, we state the following Lemma:

\begin{lemma} \label{propertyP} A necessary condition for the two bi-infinite sequence generated by \eqref{recurrence} to remain bounded is that, for all $n \in \Z$, no two consecutive terms of a triplet of the form $(u_n, v_n, v_{n-1})$ have modulus larger than 2.
\end{lemma}

\begin{proof}
The case $|v_{N-1}| \geq 2$, $|v_N| \geq 2$ has already been treated in \cite{BIST1989}, so we shall only show that the sequence $(u_n, v_n , v_{n-1})$ is unbounded assuming that there exists an index $N$ such that
\begin{equation}\label{consecutive}
|v_{N-1}| \leq 2 \: , \: |v_N| \geq 2 + \delta \:, \: |u_N| \geq 2 + \delta \quad \text{with} \quad \delta > 0.
\end{equation} 
First of all we obtain, as $|v_N| \geq 2 + \delta$,
\begin{equation} \label{first}
\begin{array}{ll}
|u_{N+1}| & \geq |u_N(v_N^2 -1)| - |v_{N-1}v_N| \\
& \geq |u_N| (|v_N^2-1| - |v_N|) \\
& \geq |u_N| (|v_N|-1)^2 \\
& \geq |u_N|(1+\delta).
\end{array}
\end{equation} 
As for the sequence $v_n$, we get,  as $|u_N \geq 2 + \delta|$,
\begin{equation}\label{second}
\begin{array}{ll}
|v_{N+1} | & \geq |u_Nv_N| - |v_{N-1}| \\
& \geq |u_N| (|v_N| -1) \\
& \geq (2 + \delta)|v_N| - (2+\delta) \\
& \geq |v_N| (1 + \delta).
\end{array}
\end{equation}
Thus, since $|v_{N+1}|$ and $|u_{N+1}|$ are obviously larger or equal than $2 + \delta$, the bounds \eqref{first} and \eqref{second} can be iterated for all indices larger than $N$ to get 
\begin{equation}
|u_{N+k}| \geq |u_N|(1+\delta)^k \; , \; |v_{N+k}| \geq |v_N|(1 + \delta)^k,
\end{equation} which proves that $|u_k|, |v_k| \overset{k}{\longrightarrow} \infty$. 
\end{proof}
\begin{remark}
The result just proven can be generalized, in fact in \cite{LaurentThesis} it is formulated, in terms of Chebyshev polynomials, for renormalization maps associated to any quasi-periodic potential of the form \eqref{potential}.
\end{remark}

Since, by \eqref{shift1}, all the points in the non-wandering set of $T_V$ must fulfill the property mentioned above, we restrict our attention to the following set:
\begin{equation} \label{union}
R_V := \{(x,y,z) \in S_V \, | \, {\bf w} (x,y,z) \; \textrm{fullfils} \; \PP \},
\end{equation}
where we say that the table ${\bf w}(x,y,z)$ fulfills the property $\PP$, when in each of the "L-shaped" subtables described above, no two adjacent boxes contain entries with absolute value larger than 2. 
It turns out that $\Sigma_V \subset \bigcap_{n= -\infty}^{\infty} T^n_V(R_V)$ by writing the set $R_V$ as the disjoint union of 10 sets, namely $R_V = \bigcup_{i=1}^{10} R_i$: defining the intervals $L^- := (-\infty, -2]$ , $s := [-2,2]$, $L^+ := [2, \infty)$, $\astc := (-\infty, \infty) = L^- \cup s \cup L^+$, we define the following sets according to which interval the entries of ${\bf w}$ lie in:
\begin{equation}\label{Markov}
\begin{array}{ll}
R_1 = \code{\astc}{s}{\astc}{s}{L^+}{s}{\astc} & R_2 = \code{\astc}{s}{\astc}{s}{L^-}{s}{\astc} \vspace{4pt} \\
R_3 = \code{s}{s}{s}{L^-}{s}{L^+}{s} & R_4 = \code{s}{s}{s}{L^+}{s}{L^-}{s} \vspace{4pt} \\
R_5 = \code{\astc}{L^+}{\astc}{\astc}{s}{s}{\astc} & R_6 = \code{\astc}{L^-}{\astc}{\astc}{s}{s}{\astc} \vspace{4pt} \\
R_7 = \code{\astc}{L^+}{s}{\astc}{s}{L^+}{s} & R_8 = \code{\astc}{L^+}{s}{\astc}{s}{L^-}{s} \vspace{4pt} \\
R_9 = \code{\astc}{L^-}{s}{\astc}{s}{L^-}{s} & R_{10} = \code{\astc}{L^-}{s}{\astc}{s}{L^+}{s}
\end{array}
\end{equation}

\begin{figure}[!ht]
\centering
\includegraphics[angle=0, width=0.8 \textwidth]{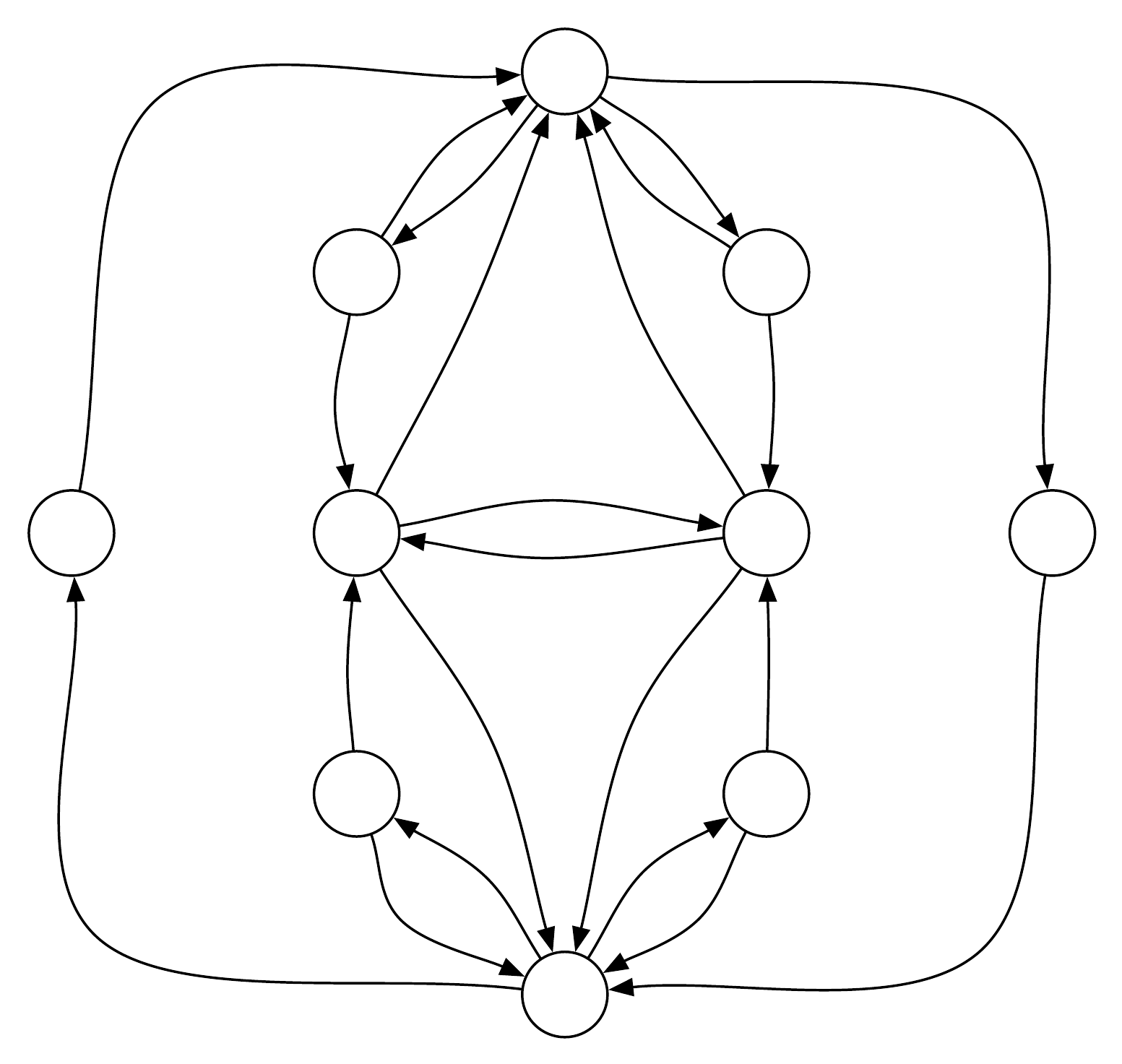}
\put(-149,253){$R_1$}
\put(-149,17){$R_2$}
\put(-203,200){$R_7$}
\put(-203,135){$R_6$}
\put(-203,68){$R_8$}
\put(-97,200){$R_{10}$}
\put(-97,135){$R_5$}
\put(-97,68){$R_9$}
\put(-25,135){$R_3$}
\put(-275,135){$R_4$}
\caption{The graph $G$: two sets $R_i, R_j$ are linked by an arrow $R_i \longrightarrow R_j$ if and only if $T(R_i) \bigcap R_j \neq \emptyset$}
\label{Graph}
\end{figure}
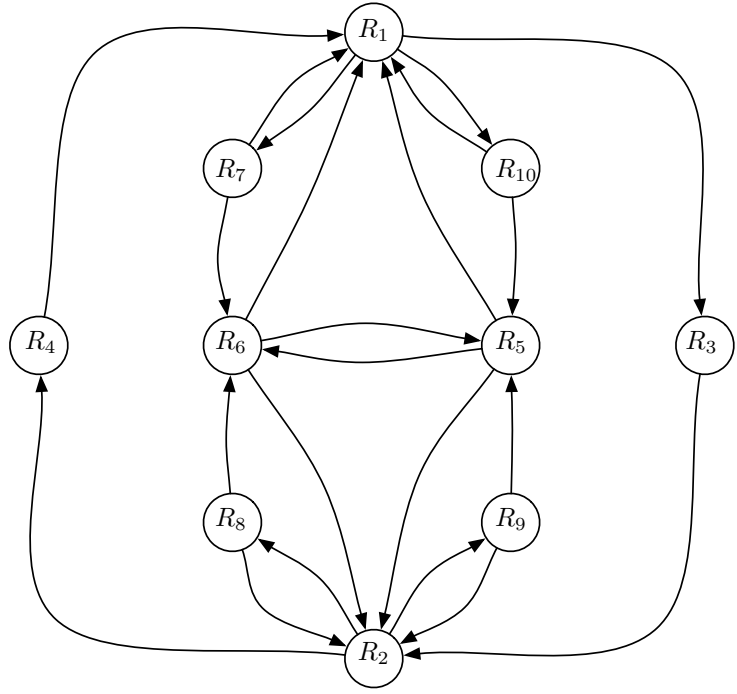

Let us now show that the sets defined by \eqref{Markov} exhaust the possible choices for our Markov partition and the graph $G$ in Figure \ref{Graph} is such that there is an arrow from the set $R_i$ to the set $R_j$ if and only if $T(R_i) \bigcap R_j \neq \emptyset$.

\begin{lemma}
The union of the ten sets defined by the tables \eqref{Markov} contains the set $R_V$ defined in \eqref{union}.
\end{lemma}

\begin{proof}
First of all let us show that the table combinations that don't appear in \eqref{Markov} define points outside of $R_V$ or no point at all. In view of Lemma \ref{propertyP}, we can exclude any table combination in which any of the L-shaped subtables contains two adiacent $L$ symbols. To exclude the tables $\code{s}{s}{s}{L^+}{s}{L^+}{s}$\; , $\code{L^+}{s}{s}{s}{s}{L^+}{s}$ \; and \; $\code{L^-}{s}{s}{s}{s}{L^+}{s}$ it is enough to notice that if $(x,y,z) \in s \times s \times L^+$ then the lower-left entry $xy-z$ has to be in $L^-$ which excludes the three cases above. The same reasoning excludes $\code{s}{s}{s}{L^-}{s}{L^-}{s}$\; , $\code{L^+}{s}{s}{s}{s}{L^-}{s}$ \; and \; $\code{L^-}{s}{s}{s}{s}{L^-}{s}$ \; , since for the $(x,y,z)$ defined by such tables the only possibility for the lower-left entry would be $L^+$. We have then showed that no table need to be added to \eqref{Markov}, and in doing so we incidentally showed that $T(R_3)$ intersects only $R_2$ while $T(R_4)$ intersects only $R_1$.

All the arrows that appear in the graph $G$ can be deduced from \eqref{Markov}, however in order to show that the graph $G$ describes \emph{all} the possible intersections between the $R_i$'s and their images, we need to "exclude" certain table combinations that don't define any point: these are obtained by suitably replacing the symbol $\ast$ in some of the tables  \eqref{Markov} by $L^-$ , $L^+$ or $s$ \footnote{remember that $\astc := (-\infty, \infty) = L^- \cup s \cup L^+$}.   
In order to show that neither $R_5$ nor $R_6$ intersect their own images by $T$, we have to show that $\code{L^+}{L^+}{\astc}{s}{s}{\astc}{\astc}$ \; and $\code{L^-}{L^-}{\astc}{s}{s}{\astc}{\astc}$ \;  define the empty set; we only formulate the argument to deal with the first table, as the one needed to deal with the second table is identical. For the top-left corner entry, we have the expression
\begin{equation}
x(y^2 -1) - zy = y(xy-z) -x.
\end{equation}
According to the table that we are considering, $x \in L^+$, $y \in s$ and $(xy-z) \in s$, and we get a contradiction with $y(xy-z) -x \in L^+$.
An identical argument prevents $T(R_7 \bigcup R_8)$ to intersect $R_5$ and $T(R_9 \bigcup R_{10})$ to intersect $R_6$.
Next we deduce that $T(R_7) \bigcap R_2 = \emptyset$ by showing that the table $\code{s}{L^+}{s}{L^-}{s}{L^+}{s}$ defines the empty set: according to such table $zy-x \in s$, $x \in L^+$ and $z \in L^+$, which implies $y > 0$. But since $(xy-z) \in L^-$ has to hold as well, we get
\begin{equation}
y(xy -z) -x \in L^-,
\end{equation}
which gives a contradiction. The very same reasoning allows us to conclude that $T(R_8 \bigcup R_9)$ does not intersect $R_1$ and $T(R_{10})$ does not intersect $R_2$. This concludes our proof as all the other arrows (or lack thereof) in the graph $G$ can be deduced by simple inspection, using the tables in \eqref{Markov}.

\end{proof}

\section{Hyperbolic Dynamics}
\subsection{Definitions: Vertical and Horizontal bases}

In order to describe the set $\Sigma_V \subset \bigcap_{-\infty}^{\infty} T^{-n}_V(R_V)$ in terms of a horseshoe dynamics, we define the following "vertical" (in a sense that will be made clear shortly) basis \begin{equation}
\mathcal B_V = \{B_1, B_2, \ldots , B_{16} \}
\end{equation} 
consisting of the sets
\begin{equation} \label{verticalbasis}
\begin{array}{cc}
B_1 = R_1 \bigcap T^{-1} (R_7) \bigcap T^{-2} (R_1) & B_2 = R_1 \bigcap T^{-1} (R_{10}) \bigcap T^{-2} (R_1) \vspace{0.1cm} \\
B_3 = R_1 \bigcap T^{-1} (R_3)  & B_4 = R_1 \bigcap T^{-1} (R_{10}) \bigcap T^{-2} (R_5) \vspace{0.1cm} \\
B_5 = R_1 \bigcap T^{-1} (R_7) \bigcap T^{-2} (R_6) & B_6 = R_2 \bigcap T^{-1} (R_4)  \vspace{0.1cm} \\
B_7 = R_2 \bigcap T^{-1} (R_9) \bigcap T^{-2} (R_2) & B_8 = R_2 \bigcap T^{-1} (R_8 )\bigcap T^{-2} (R_2)  \vspace{0.1cm} \\
B_9 = R_2 \bigcap T^{-1} (R_9) \bigcap T^{-2} (R_5) & B_{10} = R_2 \bigcap T^{-1} (R_8) \bigcap T^{-2} (R_6)   \vspace{0.1cm} \\
B_{11} = R_5 \bigcap T^{-1} (R_1) & B_{12} = R_5 \bigcap T^{-1} (R_2) \vspace{0.1cm} \\
B_{13} = R_5 \bigcap T^{-1} (R_6) & B_{14} = R_6 \bigcap T^{-1} (R_1) \vspace{0.1cm} \\
B_{15} = R_6 \bigcap T^{-1} (R_2) & B_{16} = R_6 \bigcap T^{-1} (R_5). 
\end{array}
\end{equation}
Furthermore we define the map $\Phi : \mathcal B_V \rightarrow R_V$ acting in the following way:
\begin{equation} \label{Phi}
\Phi(x) = \left\{
\begin{array}{lcr}
T^2(x) & \text{if} & x \in \bigcup_{i=1}^{10} B_i \vspace{0.1cm} \\
T(x) & \text{if} & x \in \bigcup_{i=11}^{16} B_i
\end{array}
\right.
\end{equation}
The introduction of the basis $\mathcal B_V$ and of the map $\Phi$ turns out to be appropriate as each point $x \in \Sigma_V \subset R_V$ can be written as 
\begin{equation}
x = \bigcap_{i=-\infty}^{\infty} \Phi^{-i} (B_{s_i})
\end{equation}
with $ \mathbf{s} \in \mathcal A = \left \{ \{s_i\}_{i \in \Z} \in \{1, \ldots, 16\}^{\infty} \, | \, \mathcal B_{s_i s_{i+1}} = 1 \right \}$, where the matrix $\mathcal B$ is defined as follows: 
\begin{equation}
\mathcal B = \left(
\begin{array}{cccccccccccccccc}
\rowone \\
\rowone \\
\rowtwo \\
\rowfive \\
\rowsix \\
\rowone \\
\rowtwo \\
\rowtwo \\
\rowfive \\
\rowsix \\
\rowone \\
\rowtwo \\
\rowsix \\
\rowone \\
\rowtwo \\
\rowsix
\end{array}
\right),
\end{equation}
namely the coefficients of $\mathcal B$ are chosen so that $\mathcal B_{ij} = 1 \Leftrightarrow \Phi(B_i) \bigcap B_j \neq \emptyset $ (check the definitions \eqref{verticalbasis} and the relations between the $R_i$'s depicted in Fig. \ref{Graph}). Next we introduce an "horizontal" basis of set $\mathcal B^H = \{B^1, B^2, \ldots , B^{16}\}$ defined by the relation
\begin{equation}
B^i := \Phi(B_i),
\end{equation}
and, as the last ingredient, we state the following standard definition
\begin{definition}
Denote $\widetilde S_V = R_1 \bigcup R_2 \bigcup R_5 \bigcup R_6$. We shall call a curve $\gamma$ a $\mu$-\emph{vertical curve} if $\gamma: [-2,2] \rightarrow \widetilde S_V $ is such that 
\begin{equation}
\widetilde \pi \gamma(t) = (u(t) ,t)
\end{equation}
where $\widetilde \pi : \widetilde S_V \rightarrow [-2,2] \times [-2,2]$ sends $\widetilde S_V$ bijectively into the square in the following way:
\begin{equation}
\widetilde \pi(x,y,z) := \left \{
\begin{array}{lcr}
(x,z) & \text{if} & (x,y,z) \in R_1 \bigcup R_2 \\
(y,z) & \text{if} & (x,y,z) \in R_5 \bigcup R_6 
\end{array}
\right.
\end{equation}
and the continuous function $u: [-2,2] \rightarrow [-2,2]$ satisfies
\begin{equation}
|u(t_1) - u(t_2)| \leq \mu |t_1 - t_2|.
\end{equation}
Finally, a set $B \subset \widetilde S_V$ is a $\mu$-\emph{vertical strip} if there exist two $\mu$-vertical curves $\gamma_l$, $\gamma_r$ with $\widetilde \pi \gamma_l(t) = (u_l(t), t)$ and $\widetilde \pi \gamma_r(t) = (u_r(t),t) $ such that
\begin{equation}
\widetilde \pi B = \{ (s,t) \in [-2,2] \times [-2,2]  \, | \, \ u_r(t) \leq s \leq u_r(t) \},
\end{equation}
we shall call the quantity $\text{diam}(B) := \max_{t \in [-2,2]} |u_r(t) - u_l(t)|$ the \emph{diameter} of $B$. Analogously we can define $\mu$-\emph{horizontal} curves and $\mu$-horizontal strips.
\end{definition}

\subsection{Conditions for hyperbolicity and main result}

We claim that in order to prove that $\Sigma_V$ is an hyperbolic set homeomorphic to a Cantor set, we only have to check the two following conditions:
\begin{itemize}
\item[{\bf(a)}]
For all $t = 1 \ldots 16$, the mapping  $\Phi$ takes the vertical strip $B_t \in S_V$ diffeomorphically into the horizontal strip $B^t \in S^H$, i.e.
\begin{equation}
\Phi(B_t) = B^t
\end{equation}
sending the vertical (respectively horizontal) boundaries of $B_t$ into the vertical (respectively horizontal) boundaries of $B^t$.
\item[{\bf(b)}]
Denoting by $(\xi, \zeta, \eta)$ an element of the tangent space of $S_V$ at a point $(x,y,z)$, we define the bundle of cones 
\begin{equation} \label{cones+def}
\begin{array}{c}
S_1^+ =\{ |\eta| \leq \frac{1}{3}|\xi| \}\\
S_2^+ = \{ |\eta| \leq \frac{1}{3}|\zeta| \} 
\end{array}
\end{equation}
with
\begin{equation} \label{with}
\begin{array}{c}
\pi S_1^{+} = \left(  R_1 \bigcup R_2 \right) \bigcap \left( \bigcup_{B_t \in \mathcal B_V} B_t \right) \bigcap \left( \bigcup_{B^s \in \mathcal B^H} B^s\right )\\ 
\pi S_2^+ = \left(  R_5 \bigcup R_6 \right) \bigcap \left( \bigcup_{B_t \in \mathcal B_V} B_t  \right) \bigcap \left( \bigcup_{B^s \in \mathcal B^H} B^s\right ),
\end{array}
\end{equation}
where $\pi$ denotes the canonical projection of the bundles onto their base space. Such bundles are mapped by $d\Phi$ into themselves, i.e.
\begin{equation} \label{conespreservation}
d\Phi(S_1^+ \bigcup S_2^+) \subset S_1^+ \bigcup S_2^+.
\end{equation}
Moreover if $(\xi_0, \zeta_0, \eta_0) \in S_1^+ \bigcup S_2^+$ and $(\xi_1, \zeta_1, \eta_1)$ is its image under $d\Phi$, then
\begin{equation} \label{cones}
\left\{ \begin{array}{lcr}
|\xi_1| \geq 3 |\xi_0| & \text{if} & (\xi_0, \zeta_0, \eta_0) \in S_1^+ \: , \: (\xi_1, \zeta_1, \eta_1) \in S_1^+ \\
|\zeta_1| \geq 3 |\xi_0| & \text{if} &(\xi_0, \zeta_0, \eta_0) \in S_1^+ \: , \: (\xi_1, \zeta_1, \eta_1) \in S_2^+\\       
|\xi_1| \geq 3 |\zeta_0| & \text{if} &(\xi_0, \zeta_0, \eta_0) \in S_2^+ \: , \: (\xi_1, \zeta_1, \eta_1) \in S_1^+ \\
|\zeta_1| \geq 3 |\zeta_0| & \text{if} & (\xi_0, \zeta_0, \eta_0) \in S_2^+ \: , \: (\xi_1, \zeta_1, \eta_1) \in S_2^+         
\end{array}
\right.
\end{equation}
Similarly, defining the bundles
\begin{equation}\label{cones-def}
\begin{array}{c}
S_1^- =\{ |\xi| \leq \frac{1}{3}|\eta| \} \\
S_2^- =\{ |\zeta| \leq \frac{1}{3}|\eta| \}
\end{array}
\end{equation}
respectively over the same base space as $S_1^+$ and $S_2^+$ (i.e. $\pi S_{1,2}^+ = \pi S_{1,2}^-$), we have
\begin{equation} \label{cones-pres}
d\Phi^{-1}(S_1^- \bigcup S_2^-) \subset S_1^- \bigcup S_2^-,
\end{equation}
and if $(\xi_0, \zeta_0, \eta_0) \in S_1^- \bigcup S_2^-$  and $(\xi_1, \zeta_1, \eta_1)$ is its image under $d\Phi^{-1}$, then
\begin{equation} \label{cones-stretch}
|\eta_0 | \geq 3 |\eta_1|
\end{equation}
\end{itemize}

We then have the following:

\begin{proposition} \label{symbolicok}
The $\mathcal C^1$ map $\Phi$ defined in \eqref{Phi} satisfies condition {\bf(a)} and {\bf(b)}.
\end{proposition}
The latter Proposition directly implies our main result:
\begin{theorem} \label{main}
The $\mathcal C^1$ map $\Phi$ defined in \eqref{Phi} admits the shift $\sigma : \mathcal A \rightarrow \mathcal A$, where $\mathcal A := \left\{ {\bf s} \in \{1, \ldots , 16\}^{\infty} \: | \: \mathcal B_{s_i s_{i+1}} = 1 \right \}$, and
\begin{equation}
\sigma (\mathbf{s}) = \mathbf{s}' \quad \Longleftrightarrow \quad s'_i = s_{i+1}  \quad \forall \, i \in \Z.
\end{equation} 
Namely there exists a homeomorphism $\rho$ such that for all $ {\bf x} \in \mathcal D := \bigcap_{i = -\infty}^{\infty} \Phi^i \left (\Sigma_V \right )$ we get the conjugation relation
\begin{equation}
 \Phi \circ \rho ({\bf x})= \rho \circ \sigma ({\bf x}).
\end{equation}
\end{theorem}

The proof of Theorem \ref{main} as a direct consequence of Proposition \ref{symbolicok} is a straightforward adaptation of an analogous result in \cite{Moser1973}. We refer the interested reader to the Appendix \ref{symbolicproof} for a sketch of the proof.

Let us proof Proposition \ref{symbolicok}

\begin{proof}[Proof (of Proposition \ref{symbolicok})]
We shall show how to check the two conditions for some sample strips, for the complete calculations the interested reader can check \cite{LaurentThesis}
\begin{itemize}
\item[{\bf(a)}]
Let us show that $B_6$ is a vertical strip which is sent by $\Phi$ into the horizontal strip $B^6$ so that boundaries go into boundaries, i.e. $\Phi(\partial B_6) = \partial B^6$. First of all we notice that $B_6 = R_1 \bigcap T^{-2}R_2$, hence the left boundary of $B_6$ is given by $\partial_L B_6 = R_1 \bigcap T^{-2} \partial_L R_2$, where $\partial_L R_2 = \{(-2, -E-V, E) \, , \, | E \in [-2,2] \}$ is the left boundary of $R_2$. By a straightforward calculation we get
 \begin{equation}
T^{-2}\partial_L R_2 = \{(x(E), y(E), z(E)) \, , \, E \in [-2,2]\},
\end{equation} 
where 
\begin{equation}
\left \{ \begin{array}{cl}
 x(E) & =  y(E)E+E(E+V)-2 \\
y(E) & =  (-E-V)(E^2-1)+2E \\
 z(E) & =   y(E)x(E) -E. 
 \end{array} \right.
\end{equation}
One can check that for all $V \geq 10$
\begin{equation} \label{xl13}
\left\{
\begin{array}{c}
 (x(\frac{1}{2V}) ,  z(\frac{1}{2V})) \in [-2,0] \times [-\infty, -2]\\
 (x(\frac{3}{2V}) ,  z(\frac{3}{2V})) \in  [0,2] \times [2, \infty] 
\end{array}\right. ,
\end{equation}
furthermore
\begin{equation}
V-2 <y(E) < V+2 \quad \text{for} \quad \frac{1}{2V} \leq E \leq \frac{3}{2V}  \label{yl13}.
\end{equation}
Thus, in view of the fact that in the considered interval $dz(E)/dx(E) > 1$, eqs \eqref{yl13} and \eqref{xl13} allow us to conclude that the segment 
\begin{equation}
\partial_L V_{13}  = \left\{(x(E), y(E), z(E)) \: | \: \frac{1}{2V} \leq E \leq \frac{2}{3}V  \right\} \, \bigcap \, R_1 
\end{equation} 
is a vertical curve.

Analogously we study the right boundary of $B_6$ which is given by $\partial_R B_6 = R_1 \bigcap T^{-2} \partial_R R_2$, where $\partial_R R_2 = \{(2, E-V, E) \, , \, | E \in [-2,2] \}$ is the right boundary of $R_2$. We get, for
 \begin{equation}
 \{(x(E), y(E), z(E)) \, , \, E \in [-2,2]\} = T^{-2}\partial_L R_2 
\end{equation}  
and all $V \geq 10$, 
\begin{equation} \label{xr13}
\left\{
\begin{array}{c}
(x(-\frac{3}{2V}) ,  z(-\frac{3}{2V})) \in  [-2,0] \times [-\infty, -2]  \\
 (x(-\frac{1}{2V}) ,  z(-\frac{1}{2V})) \in [0,2] \times [2, \infty]
 \end{array}\right. ,
\end{equation}
furthermore
\begin{equation}
V-2 <y(E) < V+2 \quad \text{for} \quad \frac{1}{2V} \leq E \leq \frac{3}{2V}  \label{yr13}.
\end{equation}
Thus, also in view of the fact that in the considered interval $dz(E)/dx(E) > 1$, eqs \eqref{yr13} and \eqref{xr13} allow us to conclude that the segment 
\begin{equation}
\partial_R B_6  = \left\{(x(E), y(E), z(E)) \: | \: -\frac{2}{3V} \leq E \leq -\frac{1}{2}V  \right\} \, \bigcap \, R_1 
\end{equation} 
is a vertical curve.

\begin{figure}[!ht]
\centering
\includegraphics[angle=0, width=0.6 \textwidth]{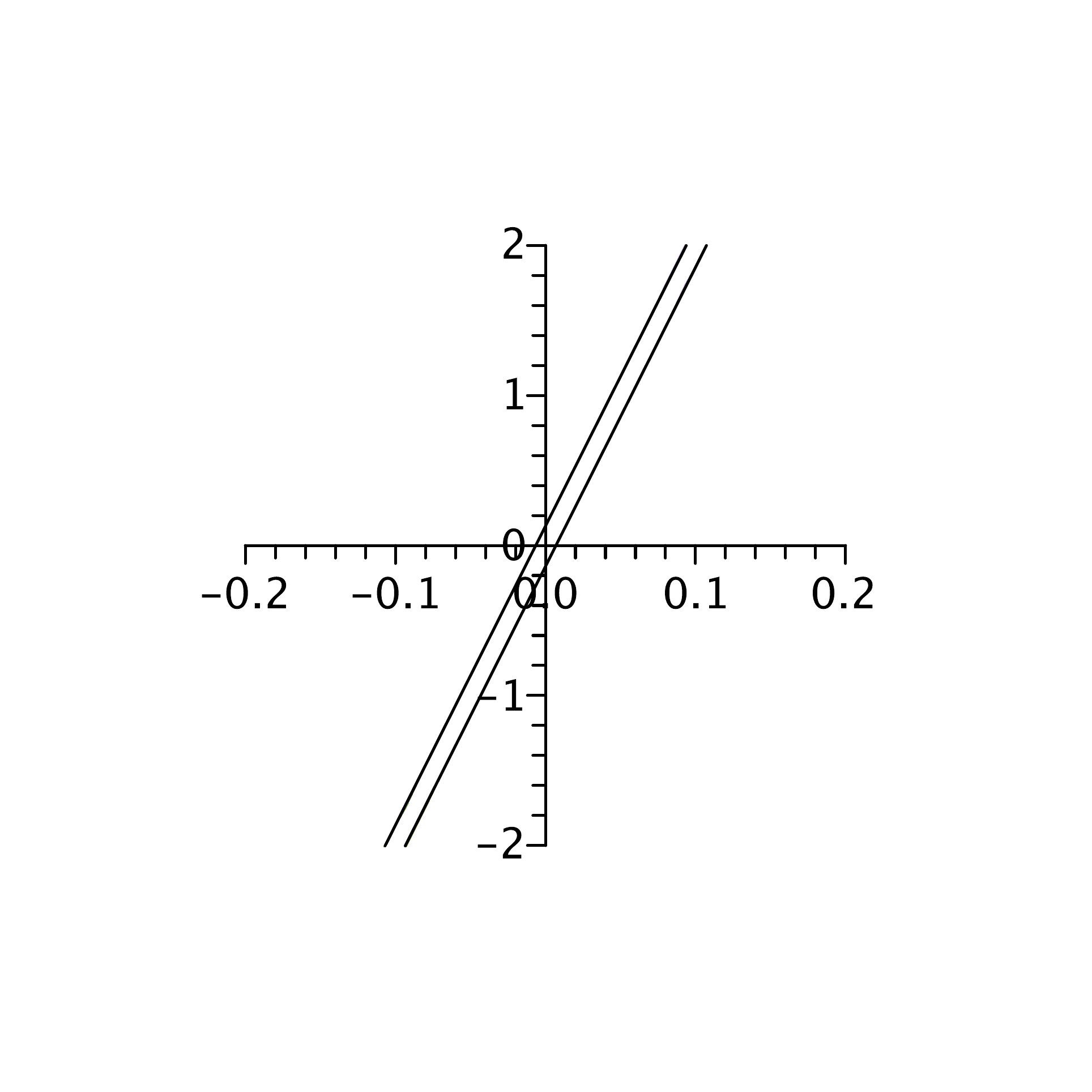}
\caption{The vertical strip $B_6$ on the $(x,z)$ plane for $V = 20$}
\label{Left13pic}
\end{figure}

In the same fashion we can prove that $B^6$ is an horizontal strip. First of all, we see that $B^6 = R_2 \bigcup T^2 R_1$. Using  $\partial_+$ and $\partial_-$ to denote upper and lower boundary, respectively, we get that $\partial_+ B^6 = \Phi^2 \partial_+ R_1$ is given by a curve $(x(E), y(E), z(E))$. Analogously to the reasoning we made above, we see that there exist $E_1 = E_1(V)$ and $E_2 = E_2(V)$, contained in the interval $[0, \frac{3}{V}]$ such that

\begin{equation} \label{B^6}
\left\{
\begin{array}{c}
(x(E_1) ,  z(E1) ) \in   [-\infty, -2] \times  \{ \frac{2}{V} \}  \\
 (x(E_2) ,  z(E_2)) \in [2, \infty] \times  \{-\frac{2}{V} \}
 \end{array}\right. ,
\end{equation}
and $-V-2 <y(E) < -V+2$ for $E \in [E_1, E_2]$. Since in the considered interval $dx(E)/dz(E) < -V$ the curve $(x(E), y(E), z(E))$ intersects $R_2$ in an horizontal curve.

\item[{\bf(b)}]

We shall first verify that, for $V$ large enough, the cones $S^-_{1,2}$ defined in \eqref{cones-def} are preserved by $d\Phi^{-1}$ as stated in \eqref{cones-pres} and stretched appropriately in the vertical direction as in condition \eqref{cones-stretch}.

First of all we notice that a tangent vector at $(x,y,z)$ is given by $\{(\xi, \zeta, \eta) \, , \, (\xi, \eta) \in \R^2  \}$ with $\zeta = \zeta(\xi, \eta)$ given by
\begin{equation} \label{tangspace}
(2x-yz)\xi + (2y-xz)\zeta + (2x-xy)\eta = 0.
\end{equation}
We compute the jacobian matrix of $T^{-1}$, denoting it $MT^{-1}$, and obtain
\begin{equation}
MT^{-1}(x,y,z) = \left(\begin{array}{ccc}-1 & z & y \\0 & 0 & 1 \\-z & z^2-1 & 2yz-x\end{array}\right).
\end{equation}
For the Jacobian matrix of $T^{-2}$, we compute
\begin{eqnarray}
MT^{-2}(x,y,z) & = &  MT^{-1}(T^{-1}(x,y,z))MT^{-1}(x,y,z) \\ 
&=& \left(\begin{array}{ccc} 1-z^2 & z(z^2-2) & 3yz^2-2xz-2y \\-z & z^2-1 & 2yz-x \\A & B & C\end{array}\right).
\end{eqnarray}
with
\begin{equation}
\begin{array}{l}
A = -2yz^4+2xz^3+4yz^2-2xz-y \\
B = 2yz^5 -2xz^4 -6yz^3 + 4xz^2 + 4yz -x \\
C = 5y^2z^4 -8xyz^3 + (3x^2 -9y^2)z^2 + 8xyz +2y^2 - x^2 -1.
\end{array}
\end{equation}
We shall only show how to treat the case of cones in $S^-_1$ based at points $(x,y,z) \in B^1\bigcup B^2 \bigcup B^3 \bigcup B^6 \bigcup B^7 \bigcup B^8$, the other cases being proved, \emph{mutatis mutandis}, in the same fashion. 

We denote an element in one of such cones as $(\xi_0, \zeta_0, \eta_0)$ and its image under $d\Phi^{-1}$ by $(\xi_1, \zeta_1, \eta_1)$, i.e. we consider the equation
\begin{equation}\label{mm}
d\Phi^{-1}(\xi_0, \zeta_0, \eta_0) = (\xi_1, \zeta_1, \eta_1).
\end{equation} 
By linearity we can assume that $\eta_0 = 1$ and $\xi_0 \in [-1/3 , 1/3]$, and we can prove the following easy lemma:
\begin{lemma}
For $V > 10$, one gets $\zeta_0 \in [-4,4]$. 
\end{lemma}

\begin{proof}
It is enough to show the Lemma for $(x,y,z) \in R_1$, the proof being identical in $R_2$. The relation \eqref{tangspace} yields
\begin{eqnarray}
\zeta_0 & = & \frac{(xy-2z)\eta_0 + (yz -2x)\xi_0}{2y-xz} \\
& \leq & \frac{4(2V+8)}{3(2V-8)},
\end{eqnarray}
and the last term is smaller than $4$ for $V > 10$. The inequality $\zeta_0 \geq -4$ is obtained in the identical way.
\end{proof}

We now give an asymptotic analysis in $V$: since $(x,y,z) \in R_1 \bigcup R_2$ the variables $x$ and $z$ vary in the interval $[-2,2]$, whereas $y$ grows linearly with $V$.

\begin{lemma} \label{orderV}
All the terms of the matrix $MT^{-2}$ are of order at most $V$ apart from $C$ which is of order $V^2$.
\end{lemma}

\begin{proof}
The claim on the terms different from $C$ is trivial, as they are at most of degree one in $y$. As for $C$, it is of order $V^2$ provided the polynomial $5z^4 - 9z^2 + 2$ does not vanish. The latter condition can be checked by using the condition that $(x,y,z) \in  B^1\bigcup B^2 \bigcup B^3 \bigcup B^6 \bigcup B^7 \bigcup B^8$ (for the precise bounds that exclude the aforementioned polynomial to vanish see \cite{LaurentThesis}).
\end{proof}

We finally get the following corollary:

\begin{corollary}
For $V$ large enough we get \eqref{cones-pres} and  \eqref{cones-stretch} for cones in $S^-_1$ based at $ B^1\bigcup B^2 \bigcup B^3 \bigcup B^6 \bigcup B^7 \bigcup B^8$.
\end{corollary}

\begin{proof}
By plugging the result of Lemma \ref{orderV} into equation \eqref{mm} we get
\begin{eqnarray}
\xi_1 & = & \mathcal O(V) \\
\eta_1 & = & \mathcal O(V^2),
\end{eqnarray} 
hence for $V$ large enough $|\xi_1| \leq \frac{1}{3} |\eta_1|$ and $|\eta_1| \geq 3$.
\end{proof}
%

We can now treat the bundles $S^+_{1,2}$. Once again we perform an asymptotic analysis for $V$ large. The preservation of the cones and the stretching condition \eqref{cones}  are more complicated  to check on such bundles than on the previous ones as we couldn't derive precise bounds on the location of the relative vertical strips. We overcomed this difficulty by making use of some ad hoc tricks. 
First of all, let us compute the Jacobian matrix of $T$ and $T^2$:
\begin{eqnarray}
MT & = & \left(\begin{array}{ccc}y^2-1 & 2xy-z & -y \\y & x & -1 \\0 & 1 & 0\end{array}\right) \\
MT^2 & = & \left(\begin{array}{ccc}A & B & C \\D & x^2(3y^2+1)-4xyz+z^2-1 & 2y(z-xy)+x \\y & x & -1\end{array}\right),
\end{eqnarray}
with
\begin{eqnarray}
A & = & 3x^2y^4 + (3y^2-1)z^2 + (4xy-6xy^3)z - (3x^2+2)y^2+1, \\
B & = & 4x^3y^3 - z^3 +6xyz^2 + (2x^2 + 2 -9x^2y^2)z - (2x^3 + 4x)y, \\ 
C & = & -3yz^2 + (6xy^2 - 2x)z -3x^2y^3 + (2x^2+2)y \\
D & = & 2x(y^3-y) + (1-2y^2)z
\end{eqnarray}

Once again we denote an element in a cone of $S^+_{1,2}$ as $(\xi_0, \zeta_0, \eta_0)$ and its image under $d\Phi$ by $(\xi_1, \zeta_1, \eta_1)$. We shall study the equation
\begin{equation}\label{vv}
d\Phi(\xi_0, \zeta_0, \eta_0) = (\xi_1, \zeta_1, \eta_1),
\end{equation} 
and check that it preserves the cones as in \eqref{conespreservation} and it fulfills \eqref{cones}.

In order to treat a first batch of cones we shall exploit a symmetry argument. In the Fibonacci case, a special symmetry allowed to avoid the analysis on the $S^+$-type of bundles altogether: in that case all the vertical strips could be seen as the image of the horizontal ones by the isometry that conjugated the renormalization map $T$ to its inverse $T^{-1}$. In our case a similar argument can be applied only to the strips $\{B_i \: | \:  i = 11 \ldots 16 \}$, which are the images of the horizontal strips $\{B^i\: | \: i = 11 \ldots 16\}$ by the isometry $\rho: (x,y,z) \mapsto (y,z,x)$: first of all we get the following:
\begin{lemma}
The following isometries hold:
\begin{equation}
\rho B^{12} = B_{14} \: , \: \rho B^{15} = B_{15} \: , \: \rho B^{11} = B_{11} \: , \: \rho B^{14} = B_{12}
\end{equation}
\end{lemma}
\begin{proof}
First of all let us observe that $\rho$ acts on the elements of the Markov partitions in the following way:
\begin{equation}
\begin{array}{ccc}
\rho R_1 = R_5 , & \rho R_2 = R_6 , & \rho R_3 = R_1 \\
\rho R_4 = R_2, & \rho R_5 = R_3, & \rho R_6 = R_4,
\end{array}
\end{equation}
from which we compute for instance
\begin{eqnarray}
\rho B^{12} & = & \rho \left( R_2 \bigcap T(R_5) \right) \\
&=&\rho R_2 \bigcap \rho T(R_5) \\
& = & R_6 \bigcap T^{-1}\rho^{-1} (R_5) \\
& =& R_6 \bigcap T^{-1}R_1 \\
&=& B_{14}
\end{eqnarray}
which proves the first equality of the claim. The other equalities follows exactly in the same way.
\end{proof}
Now, since $\rho(x,y,z) = (y,z,x)$, from the estimate we were able to recover for the variable $z$ in the computation of the horizontal strips (see again \cite{LaurentThesis}) we can get estimates on the variable $y$ on the four vertical strips under examination, precisely $y \in [-1-\frac{3}{V}, -1 + \frac{3}{V}] \cup [1-\frac{3}{V}, 1 + \frac{3}{V}]$. Then we get
\begin{lemma}
For the cones of $S^+_{2}$ based inside the set $B_{11} \bigcup B_{12} \bigcup B_{14} \bigcup B_{15}$ we have the following:
\begin{eqnarray}
\eta_1 &=&  1 \\
\xi_1 & = & \mathcal O(V)
\end{eqnarray}
\end{lemma}
\begin{proof}
It is trivial to see that the first equality holds. The multiplication of the cone $S^+_2$ with the matrix $MT$ gives $2xy$ as the dominant term in $\xi_1$, which, inside the set considered in the claim of the Lemma, is of order $V$.
\end{proof}
As we have previously checked, the above Lemma implies the desired result for the cones we are considering.
The estimates for the cones based inside the set  $B_{13} \bigcup B_{16}$ is equally simple: a direct computation leads to the estimate $\zeta_1 = \mathcal O(V)$, $\eta_1 = 1$ which in turns implies that the cones are preserved and eq. \eqref{cones}.

For the remaining 10 strips we couldn't deduce any symmetry argument of the kind explained above. Such missing symmetry is probably due to the lack of a conjugation isometry between $T$ and $T^{-1}$. We start by treating the cones of $S^+_1$ based inside the set $B_1 \bigcup B_2 \bigcup B_3 \bigcup B_6 \bigcup B_7 \bigcup B_8 $. Let us recall that we have to prove that the cone $|\eta_0| \leq |\xi_0|/3$ is send by $MT^2$ on the cone $|\eta_1| \leq |\xi_1|/3$ with the property $|\xi_1| \geq 3 |\xi_0|$.
Since we are in $R_1$ and $R_2$ the variable $y$ is the one of order $V$. The multiplication of $MT^2$ by the triplet $(\xi_0, \zeta_0, \eta_0)$ gives
\begin{equation}
\eta_1 = \mathcal O(V).
\end{equation}
For $\xi_1$ we have the following estimate
\begin{lemma}
For $\xi_1$ the following holds
\begin{equation}
\xi_1 = \mathcal O (V^2).
\end{equation}
\end{lemma}
\begin{proof}
For $x \geq \frac{1}{V}$ then the dominant term in the expression for $\xi_1$ is $x^2y^4$, which is $\mathcal O(V^2)$. For $x \leq \frac{1}{V}$ the dominant term becomes $(3z^2-2)y^2$ which is one of the terms in $A$. Since we consider only to the cones belonging to the intersections of vertical and horizontal strips (see eq. \eqref{with} in the definition of condition {\bf (b)}), we notice that $(3z^2 -2)$ never vanishes for the values of $(x,y,z)$ inside any horizontal strip contained in $R_1 \bigcup R_2$, hence
\begin{equation}
\xi_1 = \mathcal O(V^2).
\end{equation}  
\end{proof}
The last Lemma proves, in the usual fashion, the preservation of the relative cones and eq. \eqref{cones}.

The last cases left to study are the cones on the sets $B_4 \bigcup B_5 \bigcup B_9 \bigcup B_{10}$. We shall once again consider equation \eqref{vv} to check that it preserves the cones and that  $|\zeta_1| \geq 3 |\xi_0|$. It is straightforward to get 
\begin{equation}
\eta_1 = V + \mathcal O(1).
\end{equation}
For $|\zeta_1|$ we need to do some more work. We shall need the following Lemma which claims that in the cases under consideration the variables $x$ and $z$ can't be both small at the same time.

\begin{lemma}
The set $B_4 \bigcup B_5 \bigcup B_9 \bigcup B_{10}$ does not intersect the two vertical bands defined, respectively, in $R_1$ and $R_2$ by the same equation (thus by taking the two different branches of $S_V$ with respect to $y$):
\begin{equation} \label{vertforbid}
\frac{z-1/V}{V-2} \leq x \leq \frac{z+ 1/V}{V-2} \quad \text{for} \quad z \in [-2,2].
\end{equation}
This implies that 
$$
|xy-z | \gtrsim \frac{1}{V^\alpha},  \alpha>0.
$$

\end{lemma}

\begin{proof}
It suffices to show that this property is true on the edge of the four vertical strips that  we are considering. Each point of the edge of $B_4 \bigcup B_5 \bigcup B_9 \bigcup B_{10}$ can be parametrized with respect to the variable $E$, and one has the relation:
\begin{equation}\label{relbord}
z(E)=x(E)y(E)+E.
\end{equation}
 For $(x,y,z) \in B_4 \bigcup B_5 \bigcup B_9 \bigcup B_{10}$ it can been shown by a computation that the parameter $E$ belongs to   $[-1-\frac{2}{V}, -1+\frac{2}{V}]\cup [1-\frac{2}{V},1+\frac{2}{V}]$.

Using (\ref{relbord}), one has $x=\frac{z-E}{y}$ which does not belong to the vertical strip defined in \ref{vertforbid}. 

For the second part, it easy to see that $xy-z \in [-\frac{1}{V},\frac{1}{V}]$ implies $x$ to be in the vertical strip defined in (\ref{vertforbid}). One can then redo the proof replacing $\frac{1}{V}$ in (\ref{vertforbid}) by $\frac{c}{V^\alpha}$ for any positive constant $c$ and $\alpha$. This yields to 
$$
|xy-z | \gtrsim \frac{1}{V^\alpha},  \alpha>0.
$$
\end{proof}

With this lemma, we can now finish the proof and estimate $\zeta_1$.
If $x \gtrsim \frac{1}{V^\alpha}$, $0<\alpha <1$, then the leading term is $xy^3$ and applying $MT^2$ to the triplet $(\xi_0, \zeta_0, \eta_0)$ gives
$$
\zeta_1 = V^{3-\alpha} + O(V^2)
$$
which is enough to conclude,
since otherwise $x \lesssim V^{-1}$, then there is two terms of at most order $V^2$ in  D and two terms of order at most $1$ in  $2y(z-xy)+x$. We now use the lemma and the estimate  $(xy-z) \gtrsim \frac{1}{V^\alpha}$ with $\alpha >0$ valid for triplet $(x,y,z)$ in $B_4 \bigcup B_5 \bigcup B_9 \bigcup B_{10}$. 
$$
|2y^2(xy-z)| \gtrsim V^{2-\alpha}.
$$
We obtain this estimate for $\zeta_1$  
$$
\zeta_1 = V^{2-\alpha} + O(V)
$$
which, once again, allows us to conclude.

\end{itemize}
\end{proof}

\appendix
\section{}

\subsection{Sketch of the Proof of Proposition \ref{symbolicok}} \label{symbolicproof}

It suffices to show that the set
\begin{equation} \label{set}
\bigcap_{i = -\infty}^{\infty} \Phi^{-i} \left (B_{s_i} \right ) = \{ {\bf x} \in \Sigma_V \: | \: \Phi^{-n}({\bf x}) \in B_{s_n} \, , \, n= 0, \pm 1 , \pm 2, \ldots \},
\end{equation}
consists of exactly one point. In fact defining the map $\rho : \mathcal A \rightarrow \mathcal D$ as
\begin{equation}
\rho : \{\ldots s_n, \ldots s_1, s_0, s_{-1} \ldots s_{n} \ldots\} \longmapsto \bigcap_{i = -\infty}^{\infty} \Phi^i \left (B_{s_i} \right ),
\end{equation}
one easily obtains the conjugation
\begin{equation}
 \Phi \circ \rho ({\bf x})= \rho \circ \sigma ({\bf x}).
\end{equation}

\begin{figure}[!ht]
\centering
\includegraphics[angle=0, width=1 \textwidth]{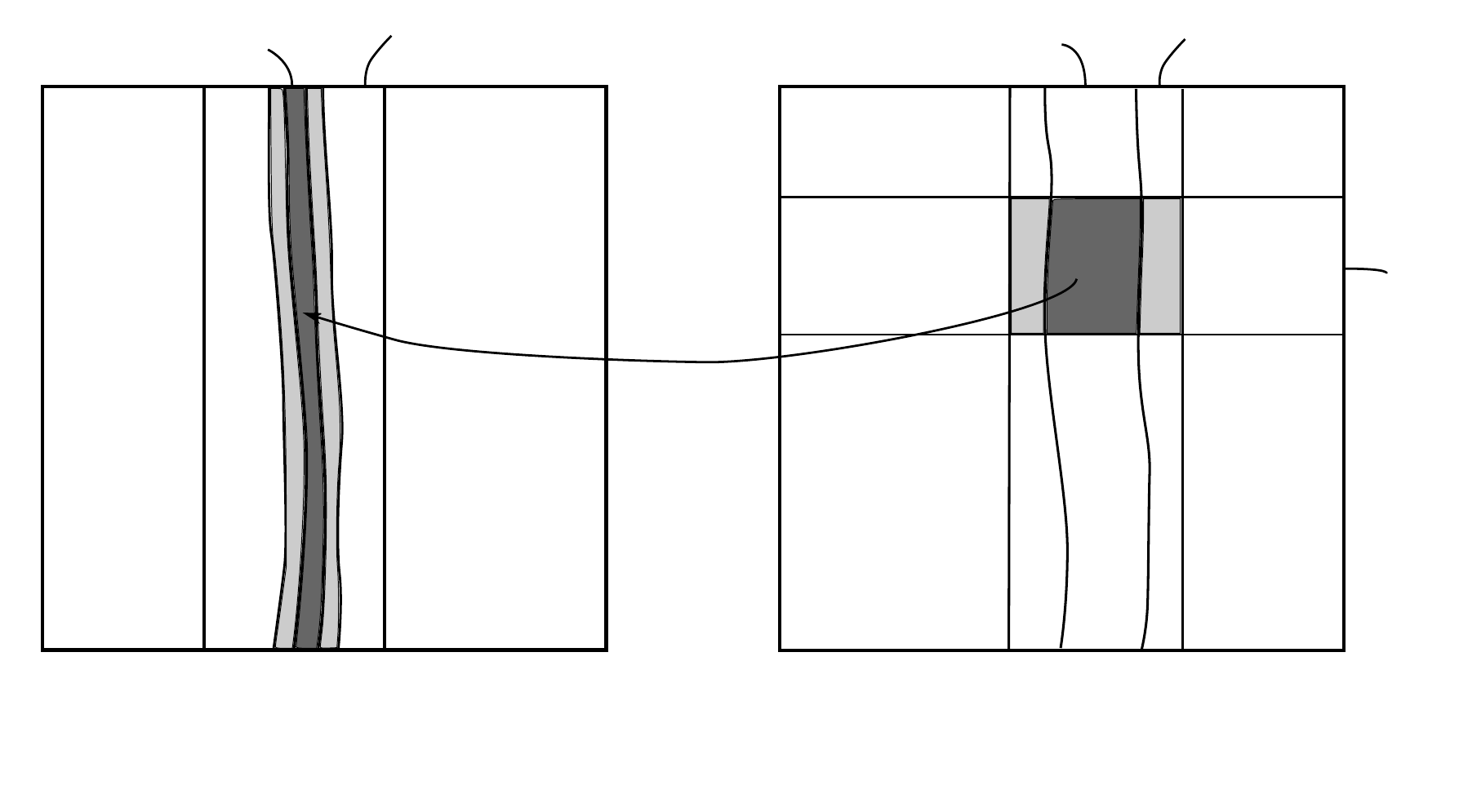}
\put(-102,192){$V$}
\put(-66,192){$B_t$}
\put(-262,192){$B_s$}
\put(-19,125){$B^s$}
\put(-87,88){$\gamma_r$}
\put(-99,100){$\gamma_l$}
\put(-94,130){$\widetilde V$}
\put(-308,190){$\widetilde B_s$}
\put(-195,115){$\Phi^{-1}$}
\caption{$\widetilde B_s := \Phi^{-1}(\widetilde V)$ is a vertical strip for all $s$ such that $\mathcal B_{ts} = 1$}
\label{Strips}
\end{figure}
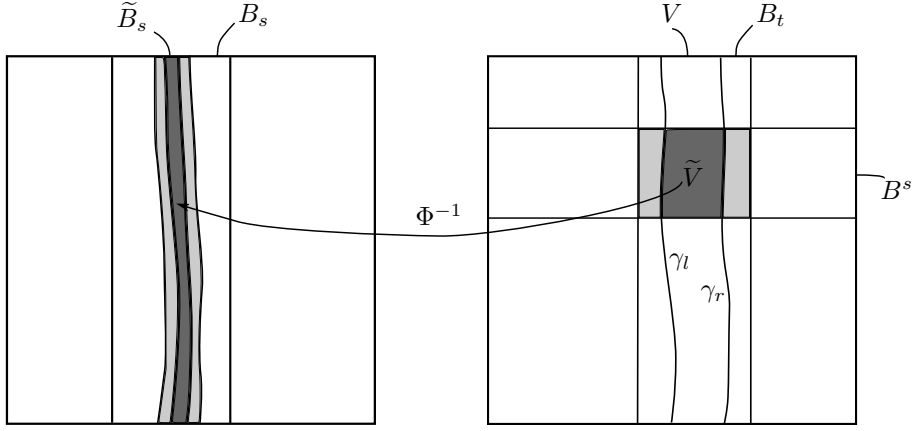

To prove that the set defined in \eqref{set} consists of exactly one point one has to show that $\mathcal U({\bf s}) := \bigcap_{i = -\infty}^{-1} \Phi^{-i} \left (B_{s_i} \right )$ is an horizontal line and $\mathcal V({\bf s}) := \bigcap_{i = 0}^{\infty} \Phi^{-i} \left (B_{s_i} \right )$ is a vertical line. It follows that the set defined in \eqref{set}, being the intersection of $\mathcal U$ and $\mathcal V$, will define precisely a point in $\Sigma_V$ (it is of course crucial the restriction to codes $\mathbf{s}$ such that $\mathcal B_{s_i s_{i+1}} = 1$ otherwise both $\mathcal U$ and $\mathcal V$ could be empty).
In order to prove that $\mathcal U({\bf s})$ and $\mathcal V({\bf s})$ are respectively an horizontal and a vertical line it is enough to prove the following condition:

\begin{itemize}
\item[$(\star)$] Given any vertical strip $B_t \in \mathcal B_V$ then, for all $V \subset B_t \in \mathcal B_V$, $\widetilde B_s := \Phi^{-1}(V) \cap B_s$ is a vertical strip for all $s  \in \{1, \ldots 16\}$ such that $\mathcal B_{ts} = 1$, furthermore
\begin{equation} \label{shrinkv}
\text{diam} (\widetilde B_t) \leq \frac{1}{2} \text{diam} (B_t)
\end{equation}
for some $0 < \nu < 1$. Analogously given any horizontal strip $B^t \in \mathcal B^H$ then for all $V \subset B^t \in \mathcal B^H$, $\widetilde B^s := \Phi^{-1}(H) \cap B^s$ is an horizontal strip for all $s \in \{1, \ldots 16\}$ such that $\mathcal B_{st}$=1, furthermore
\begin{equation}\label{shrinkh}
\text {diam} (\widetilde B^t) \leq \frac{1}{2} \text {diam} (B^t).
\end{equation}
\end{itemize}

We leave to the reader the easy task to show that from condition $(\star)$ together with (a) it follows that $\mathcal U({\bf s}) \cap \mathcal V({\bf s})$ consists of a single point, which, as discussed above, proves Proposition \ref{symbolicok}). We shall show instead that $(\star)$ is implied by (a) and (b) in Proposition \ref{symbolicok}.
Indeed let $\gamma_l$, $\gamma_r$ be, respectively, the left and right boundaries of a vertical strip $ V \subset B_t \in \mathcal B_V$, by definition $\gamma_l$ and $\gamma_r$ intersect the horizontal boundaries of $B_t$. Now let $s$, be such that $\mathcal B_{ts} = 1$, then $\gamma_l$ and $\gamma_r$ intersect $B^s$ in two vertical curves $\widehat \gamma_l$ and $\widehat \gamma_r$ respectively. These vertical curves define a vertical strip $\widehat V = V \cap B^s$.
Since $\widehat \gamma_{l,r}$ are vertical curves, we have $\widehat \gamma_{l,r}(t) = (v_{l,r}(t) , t)$ with
\begin{equation}
\left \{ 
\begin{array}{lcr}
\dot{\widehat{\gamma}}_{l,r}(t) \in S_1^- & \text{if} & \widehat \gamma \subset R_1 \bigcup R_2 \vspace{0.2cm}\\
\dot{\widehat {\gamma}}_{l,r}(t) \in S_2^- & \text{if} & \widehat \gamma \subset R_5 \bigcup R_6
\end{array}
\right.
\end{equation}
condition (b) implies that 
\begin{equation}
\left \{ 
\begin{array}{lcr}
d\Phi^{-1} (\dot{\widehat \gamma}_{l,r}) \subset S_1^- & \text{if} & \Phi^{-1}(\widehat  \gamma_{l,r}) \subset R_1 \bigcup R_2 \vspace{0.2cm}\\
d\Phi^{-1} (\dot{\widehat \gamma}_{l,r}) \subset S_2^- & \text{if} & \Phi^{-1}(\widehat \gamma_{l,r}) \subset R_5 \bigcup R_6
\end{array}
\right.
\end{equation}
which means that $\Phi^{-1} (\widehat \gamma_{l,r})$ consists of two vertical curves in $B_s$, thus $\Phi^{-1}(\widetilde V) = \Phi^{-1} (V \cap B^s) = \Phi^{-1}(V) \cap B_s$ is a vertical strip. As for \eqref{shrinkv} and \eqref{shrinkh}, take two points $\mathbf{p_l}$ and $\mathbf{p_r}$ on the vertical boundaries of $\widetilde B_s$ such that
\begin{equation}
\text{diam} (\widetilde B_s) = \Vert \mathbf{p}_r-\mathbf{p}_l \Vert, \quad \text{where} \quad \Vert (p_1,p_2) \Vert := |p_1| + |p_2|, 
\end{equation}
hence $\mathbf{p}_l$ and $\mathbf{p}_r$ lie on the straight horizontal line $\mathbf{p}(t) = t\mathbf{p}_r + (1-t)\mathbf{p}_l$ for $t \in [0,1]$. Define the curve $\mathbf{w} : [0,1] \rightarrow \widetilde V$ by
\begin{equation} \label{curve}
\mathbf{w}(t) = \Phi(\mathbf{p}(t)),
\end{equation}
so that $\mathbf{w}$ is a horizontal curve and $\mathbf{w}(0) = (\gamma_l(z_0), z_0 )$ and $\mathbf{w}(1) = (\gamma_r(z_1), z_1)$ for some $z_0, z_1 \in [-2 ,2]$. We get   
\begin{equation} \label{1}
|\gamma_r(z_1) - \gamma_l(z_0)| \leq |\gamma_r(z_1) - \gamma_l(z_1)| + |\gamma_l(z_1) - \gamma_l(z_0)| \leq \text{diam} \widetilde V + \frac{1}{3}|z_1 - z_0|
\end{equation}
furthermore, since defintion \eqref{curve} together with condition $\eqref{cones}$ imply that $\mathbf{w}(t)$ is a horizontal curve, we have
\begin{equation}\label{2}
|z_1 - z_2| \leq \frac{1}{3} |\gamma_r(z_1) - \gamma_l(z_0) |.
\end{equation}
Thus, summing \eqref{1} and \eqref{2} we obtain
\begin{equation}
\Vert \mathbf{w}(1) - \mathbf{w}(0)\Vert = |z_1 - z_0| + |\gamma_r(z_1) - \gamma_l(z_0)| \leq \frac{3}{2} \text{diam}(\widetilde V).
\end{equation}
Now, writing  \footnote{$q(t)$ lies in the $x$ or the $y$ axis according wether $B_t$ lies, respectively, inside $R_1 \bigcup R_2$ or inside $R_5 \bigcup R_6$. } $\mathbf{w}(t) = (q(t), z(t))$, we can infer from condition \eqref{cones} that $|\dot x| \geq 3 |\dot{\mathbf{p}}| > 0$, hence
\begin{equation}
\begin{array}{ll}
\text{diam} (\widetilde B_s) & = \Vert \mathbf{p}_2 - \mathbf{p}_1\Vert  = \int_0^1 |\dot{\mathbf{p}(t)}|dt  \leq \frac{1}{3} \int_0^1 |\dot x(t)|dt = \frac{1}{3}|x(1) - x(0)| \vspace{0.1cm}\\
& \leq \frac{1}{3} |\mathbf{w}(1) - \mathbf{w}(0)| \leq \frac{1}{2} \text{diam}(\widetilde V)
\end{array}
\end{equation}

\hfill $\square$

\end{document}